\documentclass{article}
\usepackage{amsmath,amsfonts,amsthm,amscd,amssymb}
\usepackage{pstricks}
\usepackage{caption}
\usepackage{graphicx}
\usepackage[colorlinks,linkcolor=red,citecolor=red,urlcolor=blue,pagebackref,hypertexnames=false]{hyperref}
\newtheorem{thm}{Theorem}[section]
\newtheorem{pro}[thm]{Proposition}
\newtheorem{cor}[thm]{Corollary}
\newtheorem{lem}[thm]{Lemma}

\newtheorem{rem}[thm]{Remark}
\newtheorem{defn}[thm]{Definition}

\begin{document}

\author{David Covert, Ye\c{s}\.im Dem\.iro\u{g}lu Karabulut, 
\\
and Jonathan Pakianathan}
\title{Cayley Digraphs Associated to Arithmetic Groups}
\maketitle

\begin{abstract} 
We explore a paradigm which ties together seemingly disparate areas in number theory, additive combinatorics, and geometric combinatorics including the classical Waring problem, the Furstenberg-S\'{a}rk\"{o}zy theorem on squares in sets of integers with positive density, and the study of triangles (also called $2$-simplices) in finite fields.  Among other results we show that if $\mathbb{F}_q$ is the finite field of odd order $q$, then every matrix in $Mat_d(\mathbb{F}_q), d \geq 2$ is the sum of a certain (finite) number of orthogonal matrices, this number depending only on $d$, the size of the matrix, and not on $q$, the size of the finite field or on the entries of the matrix.
\noindent
\


\end{abstract}

\tableofcontents

\section{Introduction and Statements of Results}

In this paper we describe a class of problems associated with some famous problems in additive and geometric combinatorics including the classical Waring problem, the Furstenberg-S\'{a}rk\"{o}zy theorem, the sum-product problem of Erd\H os and Szemer\'edi, and the distribution of triangles ($2$-simplices) in finite fields.  Throughout this paper $q$ will denote a power of an odd prime, and $Mat_n(\mathbb{F}_q)$ will denote the $n\times n$ matrices over the finite field $\mathbb{F}_q$.  $GL_n(\mathbb{F}_q)$ will denote the group of invertible matrices in $\mathbb{F}_q$.

Our basic paradigm is as follows.  Fix a group $G \subseteq GL_n(\mathbb{F}_q),$ and consider the graph $T_G = (V, \mathcal E)$ where $V = Mat_n(\mathbb{F}_q)$ and where $(A,B) \in \mathcal{E}$ if and only if $B - A \in G$.  A natural question is then to determine the diameter of $T_G$ (i.e., the maximum path length between any two vertices, if such a maximum exists) for various choices of $G$.  Here we study the Cayley Digraphs corresponding to the group of orthogonal matrices 
\[
O(n;q) = \left\{A \in Mat_n(\mathbb{F}_q): A^TA = I\right\}.
\]
Further, we show that these Cayley digraphs encode information on the congruence type and similarity type of triangles in the plane (see Section \ref{Sec:Triangles}).  We are now ready to state our results.

\begin{thm}\label{O(2)}
Every matrix in $Mat_2(\mathbb{F}_q)$ can be written as a sum of exactly eight orthogonal matrices.  If $q \equiv 3 \pmod{4}$, then every matrix in $Mat_2(\mathbb{F}_q)$ can be written as a sum of exactly six orthogonal matrices.
\end{thm}

\begin{rem}
Theorem \ref{O(2)} is sharp in general since $\left[ \begin{array}{cc}1 & 0 \\ 1 & 0 \end{array}\right] \in Mat_2(\mathbb{F}_5)$ cannot be written as a sum of seven or fewer orthogonal matrices from $O(2;5)$.  
\end{rem}

\begin{thm}\label{O(d)}
Let $d \geq 3$.  Then every matrix in $Mat_n(\mathbb{F}_q)$ can be written as a sum of $9 \cdot 6^{d-2}$ orthogonal matrices if $q \equiv 3 \pmod{4}$ and $8 \cdot 6^{d-2}$ orthogonal matrices if $q \equiv 1 \pmod{4}$.
\end{thm}

\begin{rem}
The authors wish to thank Mark Herman for pointing out that Theorems \ref{O(2)} and \ref{O(d)} are not true over arbitrary fields in general.  For example over the real numbers $\mathbb{R}$, any finite $m$-ary sumset of $O(2;\mathbb R)$ with itself is compact, and hence the set of matrices which can be written as a sum of $m$ orthogonal matrices is a compact subset of $Mat_2(\mathbb{R})=\mathbb{R}^4$ which is non-compact. Thus there is no fixed positive integer $m$ such that every $2 \times 2$ matrix over the reals is the sum of $m$ orthogonal matrices.
\end{rem}

In addition to the Spectral Gap Theorem (Theorem \ref{SpectralGap}) and Witt's Theorem (Theorem \ref{WittsThm}), some useful Lemmata in the proof of Theorem \ref{O(2)} are the following results on sums of vectors of length $1$ in $\mathbb{F}_q^d$ which are interesting in their own right.  
\begin{defn}
For $x \in \mathbb{F}_q^d$ we define $\|x \| = x_1^2 + \dots + x_d^2$, and we say $x$ is a unit vector if and only if $\|x \| =1 $.
\end{defn}
Note that $\|x \|$ is not a metric on $\mathbb{F}_q^d$, but it does preserve orthogonality as $\| x \| = x^T x$, where we regard $x \in \mathbb{F}_q^d$ as a column vector.

\begin{pro}\label{unit-vector,d=2}
Every vector in $\mathbb{F}_q^2$ is the sum of two unit vectors if and only if $q = 3$, but every vector in $\mathbb{F}_q^2$ is always the sum of $4$ unit vectors, and every nonzero vector is the sum of three unit vectors if $q \equiv 3 \pmod{4}$.  Furthermore if $q = p^n$, then the zero vector is the sum of three unit vectors if and only if $p \equiv 1,3,11 \pmod{12}$ or if $n$ is even.
\end{pro}

Notice that $\left[ \begin{array}{c}2 \\ 2 \end{array}\right] \in \mathbb{F}_5^2$ cannot be written as a sum of three or fewer unit vectors, so that Proposition \ref{unit-vector,d=2} is sharp when $q \equiv 1 \pmod{4}$.  Further notice that $\left[ \begin{array}{c}0 \\ 0 \end{array}\right] = \left[ \begin{array}{c}1 \\ 0 \end{array}\right] + \left[ \begin{array}{c}-1 \\ 0 \end{array}\right]$ is the sum of two unit vectors for all $q$.

\begin{pro}\label{unit-vector,d=3}
Every vector in $\mathbb{F}_q^3$ is the sum of $3$ unit vectors if $q \equiv 3 \pmod{4}$ and $2$ unit vectors if $q \equiv 1 \pmod{4}$.
\end{pro}

\begin{pro}\label{a^2+b^2=-1}
When $q \equiv 3 \pmod{4}$, the vector $\left[ \begin{array}{c} a \\ b \\ -1 \end{array}\right]$ cannot be written as the sum of two unit vectors when $a^2 + b^2 = -1$, though it can be written as a sum of three unit vectors.  In particular Proposition \ref{unit-vector,d=2} is best possible for all values $q$.
\end{pro}

\begin{pro}\label{unit-vector,d=4}
Every vector in $\mathbb{F}_q^d$ is the sum of $2$ unit vectors for all $d \geq 4$.
\end{pro}

Our final result concerns triangles in $\mathbb{F}_q^2$.  In general, a $n$-simplex in $\mathbb{F}_q^n$ is an ordered set of $n+1$ vectors $[v_0, \dots , v_n]$ with $v_j \in \mathbb{F}_q^n$.  A $2$-simplex is called a triangle.  An $n$-simplex $[v_0, \dots, v_n]$ is called nondegenerate if 
\[
\dim\left(\operatorname{span}\left\{(v_j - v_0) : 1 \leq j \leq n\right\}\right) = n.
\]
For $b \in \mathbb{F}_q^n$ let $T_b(x) = x + b$ be the translation map by $b$.  Let $T = \{T_b : b \in \mathbb{F}_q^n\}$ be the set of all translations.  The group generated by $T$ and $O(n;q)$ is called the Euclidean group which we denote $E(n;q)$.  Let $\psi_{A,b} \in E(n;q)$ be given by $\psi_{A,b}(x) = Ax + b$. The Euclidean group $E(n;q)$ then acts on the set of $n$-simplices via the group action
\[
\psi_{A,b}([v_0, \dots , v_n]) = [\psi_{A,b}(v_0), \dots , \psi_{A,b}(v_n)].
\]
The orbits of this group action are called congruence classes.  In other words, two simplices $t = (x_0, \dots , x_n)$ and $t' = (y_0, \dots , y_n)$ are congruent if and only if there exists an orthogonal matrix $A \in O(n;q)$ and a vector $b \in \mathbb{F}_q^n$ such that $x_i = A y_i + b$ for $i = 0, \dots ,n$.

\begin{thm}\label{triangles}
The number of congruence classes of nondegenerate triangles in the plane $\mathbb{F}_q^2$ is given by
\[
\left\{
\begin{array}{ccc}
\frac{q(q^2-1)}{2} && q \equiv 1 \pmod{4}
\\
\\
\frac{q(q-1)^2}{2} && q \equiv 3 \pmod{4}
\end{array}
\right.
\]
\end{thm}

\section{Background}

In 1770 Waring asserted that every integer can be written as a sum of $4$ squares, $9$ cubes, $19$ biquadrates, and so on (\cite{VW}). In the context of Cayley graphs, when $G \subseteq GL_1(\mathbb{F}_q) = \mathbb{F}_q^*$, $G$ must be a power subgroup $G=(\mathbb{F}_q^*)^k$. In this case the diameter of the Cayley graph $T_G$ is the same as the minimum number $m$, such that every element of $\mathbb{F}_q$ is a sum of $m$ many $k$th powers. The determination of the diameter of $T_G$ is thus a variant of Waring's problem over $\mathbb{F}_q$.  The solution of the Waring problem for general (not necessarily commutative) finite rings was obtained by the second listed author in \cite{Dem2017a} in part by studying this graph and its spectrum.

Recall that the classical Furstenberg-S\'{a}rk\"{o}zy Theorem (\cite{Furstenberg,Sarkozy}) states that every set of integers with positive (natural) density contains two elements whose difference is a square.  The spectral theorem applied to these power subgroups $G = (\mathbb{F}_q^*)^k$ graphs also yields (\cite{Dem2017a}) a finite field generalization of the Furstenberg-S\'{a}rk\"{o}zy Theorem for all powers $k \geq 2$:  if $E \subseteq \mathbb{F}_q$ satisfies $|E| \gg_k \sqrt{q}$ then there exists $e, e' \in E, a \in \mathbb{F}_q$ with $e-e' = a^k.$

The question of determining the minimum number, say $m$, of units such that every element is a sum of $m$ units is well known (see \cite{Srivastava}, for example).  In the context of Cayley graphs, when $G=GL_n(\mathbb{F}_q)$, the Cayley digraph $T_{GL_n}$ is called the ``unit-graph'' and can be considered as an undirected graph as $x \in G \iff -x \in G$.  It was shown in \cite{Dem2017b} that outside of the case $n=1, q=2$, every matrix is the sum of two invertible matrices. The author then used this result to recover the classical result (\cite{Henriksen}) that in any finite ring of odd order, every element is the sum of two units (\cite{Dem2017b}).

For a subset of a ring $A \subseteq R$, we define the sum set, product set, and difference set of $A$ as
\[
A+A = \{a + b : a,b \in A\},
\]
\[
A \cdot A = \{a \cdot b : a,b \in A\},
\]
and
\[
A-A = \{a - b : a,b \in A\},
\]
respectively.  The sum-product conjecture of Erd\H os-Szemer\'edi asserts that for any $A \subseteq \mathbb{Z}$, either the sum set or the product set must be large in the sense that for all $\varepsilon > 0$, there exists $C(\varepsilon)$ such that
\[
\max (|A+A|,|A\cdot A|) \geq C(\varepsilon) n^{2 - \varepsilon},
\]
where the maximum is taken over all sets $A \subseteq \mathbb{Z}$ with cardinality $|A| = n$.
Analogues of this conjecture have long been considered in finite fields and various rings (\cite{HIS, Tao}).  When the spectral gap theorem is applied to the group $G = SL_n(\mathbb{F}_q)$, one obtains a sum-product type result similar to those considered in \cite{HI}. 
\begin{thm}[\cite{Dem2017b}]
For a set $A \subseteq \mathbb{F}_q$, one has $(A-A)(A-A) = \mathbb{F}_q$ so long as $|A| \geq \frac{3}{2}q^{3/4}$.
\end{thm}
\noindent It was also shown in \cite{Dem2017b} that as long as $n \geq 2$, then every matrix is the sum of two invertible matrices of determinant one, or in other words, that this graph has diameter two.

\subsection{Triangles and nondegenerate simplices}
The classical Erd\H os-distance problem posed in 1946 (\cite{Erdos46}) asks one to determine the minimum number of distances determined by $n$ points in the plane.  Before a full resolution of the distance problem was achieved, a finite field analogue of the distance problem was first considered by Bourgain, Katz, and Tao (\cite{BKT}) with the modern formulation of the finite field distance problem being due to Iosevich and Rudnev (\cite{IR}).  Interestingly, the ``harder" problem of the classical Erd\H os distance problem has been solved, while the finite field problem remains open.  To gain insight into these distance problems, it is natural to consider generalizations of the distance problem, such as that of studying simplices.  Recall that an $n$-simplex is simply an ordered list $t = [v_0, \dots , v_n]$ of vectors in $\mathbb{F}_q^n$, and we say that $t$ is nondegenerate if $\operatorname{span}(x_j - x_0 : 1 \leq j \leq n)$ is $n$-dimensional.  An $n$-simplex yields a tuple of ${n + 1 \choose 2}$  distances between pairs of vertices, and it is well known (\cite{BIP}) that two nondegenerate simplices are congruent if and only if they have the same distance type.  Therefore, the set of all $n$-simplices can naturally be viewed to be ${n + 1 \choose 2}$-dimensional.

Up to translation every $n$-simplex is congruent to one whose first vertex is the zero vector, and we call these $n$-simplices pinned at zero.  To any such $n$-simplex pinned at the origin, say $t = [0, x_1, \dots , x_n]$, we may associate a unique matrix in $M \in Mat_n(\mathbb{F}_q)$ so that $\operatorname{col}_j(M) = x_j$ for $j = 1 , \dots , n$.  This association is easily seen to be a bijection between the $n$-simplices pinned at zero and $Mat_n(\mathbb{F}_q)$, and the left action of $O(n;q)$ on the $n$-simplices pinned at zero amounts to the left multiplication action of $O(n;q)$ on $Mat_n(\mathbb{F}_q)$.  Now it is straight-forward to check that if $t_1, t_2$ are two matrices corresponding to two $n$-simplices, then the two $n$-simplices are congruent if and only $O(n;q)t_1 = O(n;q)t_2$, as the only elements in $E(n;q)$ that  which take the origin to itself are the orthogonal transformations $O(n;q)$.   We combine our observations and record them in the following proposition.
\begin{pro}\label{pro:trianglefacts}
Left multiplication action of $O(n;q)$ on $Mat_n(\mathbb{F}_q)$ decomposes $Mat_n(\mathbb{F}_q)$ into orbits of the form $O(n;q)t$, where $t \in Mat_n(\mathbb{F}_q)$.  The orbits $O(n;q)t$ are in one-to-one correspondence with the $n$-simplices in $\mathbb{F}_q^n$.  In particular the right cosets of $O(n;q)$ in $GL_n(\mathbb{F}_q)$ correspond exactly to the congruence classes of non-degenerate $n$-simplices in $\mathbb{F}_q^n$, and the number of distinct congruence class of nondegenerate $n$-simplices in $\mathbb{F}_q^n$ is the index of $O(n;q)$ in $GL_n(\mathbb{F}_q)$.
\end{pro}

\section{Preliminary Lemmas}

The adjacency matrix of the Cayley digraph $T_G$ is the matrix whose $A-B$ entry is $1$ if $(A,B) \in \mathcal E$ (i.e., if $B-A \in G$) and $0$ otherwise, considered as a $q^{n^2} \times q^{n^2}$ matrix (with respect to some ordering of $Mat_n(\mathbb{F}_q)$) . The multiset of eigenvalues of this adjacency matrix is called the spectrum of $T_G$.  In general this adjacency matrix is not symmetric, hence these eigenvalues need not be real and in general are complex.  Though adjacency matrices of digraphs are not in general diagonalizable, the adjacency matrices of Cayley digraphs based on abelian groups are diagonalizable as they exhibit an orthogonal basis of eigenvectors coming from the characters of the underlying abelian group.  As a result we have a nice spectral gap theorem.

\begin{thm}[Spectral Gap Theorem]\label{SpectralGap}
Let $T_G$ be the Cayley digraph with vertex set $V = Mat_n(\mathbb{F}_q) $ where there is an edge between $A$ and $B$ if and only if $B - A \in G$.  Let $\{\chi_i\}$ denote the set of additive characters on $G$, and define
\[
n_* = \frac{q^{n^2}}{|G|}  \max_{i} \left| \sum_{g \in G} \chi_i(g)\right|.
\]
Then $T_G$ has an $X-Y$ edge if $\sqrt{|X||Y|} > n_*$, and if $|X| > n_*$, then there exists two distinct vertices $x, y \in X$ such that there is an edge from $x$ to $y$.
\end{thm}
We also rely on Witt's Extension Theorem (\cite{Lang}).
\begin{thm}[Witt's Theorem]\label{WittsThm}
Let $q$ be odd, and suppose $B$ is a bilinear form on $\mathbb{F}_q^d$.  Then every isometry between two subspaces of $(\mathbb{F}_q^d, B)$ can be extended to an isometry on all of $(\mathbb{F}_q^d, B)$.  In particular if $x, y \in \mathbb{F}_q^d \setminus \{ \vec 0 \}$, and if $\| x \| = \| y \|$, then there is an isometry mapping $x$ to $y$.
\end{thm}

Another concept that is useful in determining both the diameter and spectrum of the Cayley digraph $T_G$ is that of $G$-equivalence of matrices.

\begin{defn}
Let $G \subseteq GL_n(\mathbb{F}_q)$ be a fixed linear group. Two matrices $A, B \in GL_n(\mathbb{F}_q)$ are $G$-equivalent if and only if there exist $x, y \in G$ such that 
$B=xAy$. This is easily checked to be an equivalence relation and its equivalence classes are of the form $GtG$ where $t \in Mat_n(\mathbb{F}_q)$.
\end{defn}
For example, two matrices $A$ and $B$ are $GL_n(\mathbb{F}_q)$-equivalent if and only if $A$ and $B$ have the same rank. On the other hand, $A$ and $B$ are $SL_n(\mathbb{F}_q)$-equivalent if and only if they have the same determinant and rank (\cite{Dem2017b}). The notion of $G$-equivalence will be useful for our purposes due to the following Propositions.

\begin{pro}
Suppose that $G \subseteq GL_n(\mathbb{F}_q)$ is a fixed linear group, and suppose $A$ and $B$ are $G$-equivalent. Then $A$ can be written as a sum of $m$ elements of $G$ if and only if $B$ can be written as a sum of $m$ elements of $G$.
\end{pro}
\begin{proof}
Suppose $A = g_1 + \dots + g_m$ where $g_j \in G$, i.e., $A$ is the sum of $m$ elements of $G$. Now $B=xAy$ for some $x, y \in G$ so 
$B=xg_1y + \dots + xg_my=g_1' + \dots g_m' $ where $g_j'=xg_jy \in G$. Thus $B$ is also the sum of $m$ elements of $G$.

\end{proof}
Let $\chi(\cdot )$ be a nontrivial character of $\mathbb{F}_q$ (typically we just use the canonical additive character $\chi(x) = \exp(2\pi i Tr_G(x)/p)$, where $p$ is the characteristic of $\mathbb{F}_q$ and where $Tr_G$ is the Galois trace).  Then the set of all characters on $Mat_n(\mathbb{F}_q)$ is given by $\left\{\chi(Tr_M(Ax)) : A \in Mat_n(\mathbb{F}_q)\right\}$, where $Tr_M$ denotes the matrix trace of $A \in Mat_n(\mathbb{F}_q)$.  For convenience we may write $\chi(Tr_M(Ax)) = \chi_A(x)$. Note that each of these characters $\chi_A$ is an eigenfunction of the adjacency matrix of $T_G$ with corresponding eigenvalue
\[
\lambda_A = \sum_{g \in G} \chi_A(g).
\]
Moreover the eigenvalues of $T_G$ corresponding to $G$-equivalent matrices are equal.
\begin{pro}
Suppose that $A$ and $B$ are $G$-equivalent.  Then $\lambda_A = \lambda_B$.
\end{pro}
\begin{proof}
We have $A = xBy$ for some $x, y \in G$.  Hence
\[
\lambda_A = \sum_{g \in G} \chi(Tr_M(Ag)) = \sum_{g \in G} \chi(Tr_M(xByg)) = \sum_{g \in G} \chi(Tr_M(Bygx)) = \lambda_B
\]
ad $g \mapsto ygx$ is a bijection from $G$ to itself.
\end{proof}

\subsection{Triangle Lemmas}\label{Sec:Triangles}

We further require the following Lemmas on triangles. Recall from Proposition \ref{pro:trianglefacts} that a matrix $t \in GL_2(\mathbb{F}_q)$ encodes a nondegenerate triangle pinned at zero in the plane $\mathbb{F}_q^2$, and that $O(2)t=O(2)t'$ if and only if the triangles $t$ and $t'$ are congruent.  We first develop an alternate condition for this congruence that will be useful in our calculations.  Put $t=\begin{bmatrix} a & b \\ c & d \end{bmatrix}$, and let $L_1(t)=a^2+c^2, L_2(t)=b^2+d^2$, and $\mu(t)=ab+cd$ denote the length of the first column, the length of the second column, and the dot product between the columns, respectively.

\begin{lem}
\label{lem: grungy1}
 Fix $q$ an odd prime power. 
Let $\begin{bmatrix} a \\ c \end{bmatrix}$ be a vector in the plane of nonzero length $L_1=a^2+c^2 \neq 0$. Then given $L_2, \mu \in \mathbb{F}_q$, there exists $s$ vectors 
$\begin{bmatrix} b \\ d \end{bmatrix}$ such that $L_2=b^2+d^2$ and $\mu=ab+cd$, where
\[
s = \left\{ 
\begin{array}{ccc}  
2 && L_1L_2 - \mu^2 \text{ is a nonzero quadratic residue} 
\\
1 && L_1L-2 - \mu^2 = 0
\\
0 && \text{ otherwise}
\end{array} \right.
\]
When $s > 0$, the solution(s) are given by: 
$$
\begin{bmatrix} b \\ d \end{bmatrix} = \frac{\mu}{L_1} \begin{bmatrix} a \\ c \end{bmatrix} \pm \frac{\sqrt{L_1L_2-\mu^2}}{L_1}\begin{bmatrix} -c \\ a \end{bmatrix}.
$$
Note in this case that the solutions for $\begin{bmatrix} b \\ d \end{bmatrix}$ are linearly independent from $\begin{bmatrix} a \\ c \end{bmatrix}$ if and only if $L_1L_2-\mu^2 \neq 0$.

\end{lem}
\begin{proof}
Note $\begin{bmatrix} a \\ c \end{bmatrix} \neq \begin{bmatrix} 0 \\ 0 \end{bmatrix}$ and so $\begin{bmatrix} b \\ d \end{bmatrix}$ is forced to live on the affine line (line not necessarily through origin) $ab+cd=\mu$ in order to satisfy the dot product constraint.
Thus $\begin{bmatrix} b \\ d \end{bmatrix}=t \begin{bmatrix} -c \\ a \end{bmatrix}+\begin{bmatrix} 0 \\ c^{-1} \mu \end{bmatrix}$ for some $t$ in the case $c \neq 0$ or $\begin{bmatrix} b \\ d \end{bmatrix}=\begin{bmatrix} a^{-1}\mu \\ t \end{bmatrix}$ for some $t$ in the case $c=0$. 

We finish the proof in the first case, the second case when $c=0$ being similar and easier is left to the reader (you will use that $L_1=a^2+c^2=a^2$ is a nonzero quadratic residue in this case).

Plugging this case into the column two length constraint $b^2+d^2=L_2$ we get $t^2c^2+(ta+c^{-1}\mu)^2=L_2$. Using $L_1=a^2+c^2$ this becomes
$$L_1t^2 +2tac^{-1}\mu + \mu^2/c^{-2} -L_2 = 0.$$ The number of solutions in $t$ this has in the field corresponds to the number of solutions to $\begin{bmatrix} b \\ d \end{bmatrix}$. 
Using the quadratic equation, this depends on whether the discriminant of the quadratic, $(2ac^{-1}\mu)^2 - 4L_1(\mu^2/c^{-2}-L_2)$ is a nonsquare, zero or nonzero quadratic residue in the field (no solution, one solution, two solutions respectively). This discriminant simplifies to $(4/c^2)(a^2\mu^2-L_1\mu^2)+4L_1L_2=-4\mu^2+4L_1L_2$ and so is a nonsquare, zero or nonzero quadratic residue exactly when $L_1L_2-\mu^2$ is. It also follows that when there is a solution, $t$ is given by
$$
t=\frac{-2ac^{-1}\mu \pm \sqrt{4(L_1L_2-\mu^2)}}{2L_1}=\frac{-ac^{-1}\mu \pm \sqrt{L_1L_2-\mu^2}}{L_1}
$$
Thus 
$$
\begin{bmatrix} b \\ d \end{bmatrix} = t \begin{bmatrix} -c \\ a \end{bmatrix} + \begin{bmatrix} 0 \\ c^{-1} \mu \end{bmatrix}
$$
becomes
$$
\begin{bmatrix} b \\ d \end{bmatrix} = \frac{\mu}{L_1}\begin{bmatrix} a \\ c \end{bmatrix} \pm \frac{\sqrt{L_1L_2-\mu^2}}{L_1} \begin{bmatrix} -c \\ a \end{bmatrix}
$$

\end{proof}

The last proof goes through when $\begin{bmatrix} a \\ c \end{bmatrix} \neq \begin{bmatrix} 0 \\ 0 \end{bmatrix}$ has zero length $L_1=0$ also (note this means $a$ and $c$ have to both be nonzero), except that the quadratic obtained in the last proof reduces to a linear equation instead. We record this in the following Lemma and leave the proof to the reader:

\begin{lem}
\label{lem: grungy2}
Let $q$ be an odd prime power and let $\begin{bmatrix} a \\ c \end{bmatrix} \neq \begin{bmatrix} 0 \\ 0 \end{bmatrix}$ be a nonzero vector with zero length $L_1=a^2+c^2=0$. Given $L_2$ and $\mu$ in $\mathbb{F}_q$ there exists a unique vector 
$\begin{bmatrix} b \\ d \end{bmatrix}$ such that $b^2+d^2=L_2$ and $ab+cd=\mu$ as long as $\mu \neq 0$. 
When $\mu=0$ and $L_2=0$ there are $q$ solutions. When $\mu=0$ and $L_2 \neq 0$ there is no solution for $\begin{bmatrix} b \\ d \end{bmatrix}$ with these properties.
\end{lem}

Note $\begin{bmatrix} a \\ c \end{bmatrix}$ is a nonzero vector of zero length if and only if $a^2=-c^2$ or equivalently if and only if $-1=(a/c)^2$. These hence can only exist when $-1$ is a quadratic residue in the field, i.e., when 
$q \equiv 1 \pmod 4$. In this case if $i$ is a primitive $4$th root of unity in the field, $a/c=\pm i$. Thus these vectors of length zero consist of the elements on the union of the two lines 
spanned by $(1,i)$ and $(1,-i)$ in the plane. Each of these lines is isotropic, i.e. equal to its own orthogonal under the dot product.  The conditions $L_1=0$ and $\mu=0$ hence force the vector $\begin{bmatrix} b \\ d \end{bmatrix}$ to lie on the same isotropic line as $\begin{bmatrix} a \\ c \end{bmatrix}$ and hence $L_2$ must be zero also which is consistent with the last Lemma.  We are now ready to prove the following Theorem.

\begin{thm}
\label{thm: nondegenerate triangle congruence}
Let $t, t' \in GL_2$, then $t$ and $t'$ determine congruent (nondegenerate) triangles if and only if $L_1(t)=L_1(t'),L_2(t)=L_2(t')$ and $\mu(t)=\mu(t')$.
Furthermore we must have $L_1L_2-\mu^2$ be a nonzero quadratic residue in $\mathbb{F}_q$.
\end{thm}
\begin{proof}
One direction is clear. We only need to prove the nonobvious direction i.e., if $L_1(t)=L_1(t'), L_2(t)=L_2(t'), \mu(t)=\mu(t')$ then $t$ is congruent with $t'$.

As we are talking about nondegenerate triangles, the first column of both $t$ and $t'$ are nonzero vectors of the same length. By a theorem of Witt, it follows that there is an orthogonal 
matrix $x$ taking the 1st column of t' to the 1st column of t. Thus we may assume $t$ and $t'$ share the same first column 
$\begin{bmatrix} a \\ c \end{bmatrix}$ with length $L_1=a^2+c^2$. It follows from the existance of either $t$ or $t'$ that there is a 2nd column vector 
$\begin{bmatrix} b \\ d \end{bmatrix}$ with length $L_2=b^2+d^2$ and $\mu=ab+cd$. However by Lemmas~\ref{lem: grungy1} and \ref{lem: grungy2}, this happens with $\begin{bmatrix} b \\ d \end{bmatrix}$ linearly independent from $\begin{bmatrix} a \\ c \end{bmatrix}$ if and only if $L_1L_2-\mu^2$ is a nonzero quadratic residue (using Lemma~\ref{lem: grungy1} when $L_1 \neq 0$ or using Lemma~\ref{lem: grungy2} when $L_1=0$ and $\mu \neq 0$.)
(Note the case of $L_1=L_2=\mu=0$ consists of both vectors lying on the same isotropic line in which case the matrix does not have rank 2 and is hence degenerate which we exclude).

The cases when there are two solutions for $\begin{bmatrix} b \\ d \end{bmatrix}$ only occur when $L_1L_2-\mu^2$ is a nonzero quadratic residue and $L_1 \neq 0$. In this case there are two possible solutions for $\begin{bmatrix} b \\ d \end{bmatrix}$ given by 
$\begin{bmatrix} b \\ d \end{bmatrix} = \frac{\mu}{L_1}\begin{bmatrix} a \\ c \end{bmatrix} \pm \frac{\sqrt{L_1L_2-\mu^2}}{L_1} \begin{bmatrix} -c \\ a \end{bmatrix}$. Thus there are two possible matrices in $GL_2(\mathbb{F}_q)$ of the form $\begin{bmatrix} a & b \\ c & d \end{bmatrix}$ with the given column lengths and column dot product and fixed first column. By assumption $t, t'$ come from this set of two matrices so to be done, we just have to show that these two matrices determine congruent triangles.

Given the explicit formula above for $\begin{bmatrix} b \\ d \end{bmatrix}$ it is easy to check that the matrix $\begin{bmatrix} \frac{a^2-c^2}{L_1} & \frac{2ac}{L_1} \\ \frac{2ac}{L_1} & \frac{c^2-a^2}{L_1} \end{bmatrix}$ is in $O(2;q)$ and takes 
$\begin{bmatrix} a \\ c \end{bmatrix}$ to $\begin{bmatrix} a \\ c \end{bmatrix}$, $\begin{bmatrix} -c \\ a \end{bmatrix}$ to $\begin{bmatrix} c \\ -a \end{bmatrix}$ and hence switches the two solutions given for $\begin{bmatrix} b \\ d \end{bmatrix}$ above. Thus indeed the two matrices above determine congruent triangles and we are done.

\end{proof}

From Theorem~\ref{thm: nondegenerate triangle congruence}, it is easy to determine when two invertible matrices $t$ and $t'$ determine the same right coset of $O(2;q)$ in $GL_2(\mathbb{F}_q)$, or equivalently determine congruent triangles. It is if and only if their column lengths and dot product is the same. Furthermore these must satisfy the compatibility condition that 
$L_1L_2-\mu^2$ is a nonzero quadratic residue in the field. Thus the congruence classes of nondegenerate triangles can be parametrized exactly by triples 
$(L_1, L_2, \mu) \in \mathbb{F}_q^3$ subject to this compatibility condition.  The left cosets of $O(2;q)$ in $GL_2(\mathbb{F}_q)$ work similarly with rows replacing columns throughout. 

Finally note that if one has a nondegenerate triangle in the plane $\mathbb{F}_q^2$ with vertices at the origin, $x=\begin{bmatrix} a \\ c \end{bmatrix}$ and $y=\begin{bmatrix} b \\ d \end{bmatrix}$ with $L_1=a^2+c^2, L_2=b^2+d^2, \mu=ab+cd$, then the length of the last side of the triangle, $L_3$, is given by $L_3=(x-y) \cdot (x-y)=x \cdot x + y \cdot y - 2x \cdot y = L_1 + L_2 - 2 \mu$ and so $\mu=\frac{L_1+L_2-L_3}{2}$. Plugging this into Theorem~\ref{thm: nondegenerate triangle congruence} yields the following result also appearing in \cite{BIP}.
\begin{cor}
\label{cor:  lengthtype condition for triangles}
Fix $q$ an odd prime power. There exists a nondegenerate triangle with side lengths $L_1, L_2, L_3$ in $\mathbb{F}_q^2$ if and only if $2\sigma_2-P_2$ is a nonzero quadratic residue. 
Here $P_2=L_1^2+L_2^2+L_3^2$ and $\sigma_2=L_1L_2+L_1L_3+L_2L_3$ are symmetric polynomials in $L_1, L_2, L_3$.
\end{cor}
\begin{proof}
$$L_1L_2-\mu^2 = \frac{4L_1L_2 - (L_1+L_2-L_3)^2}{4} = \frac{2\sigma_2-P_2}{4}$$
is a nonzero quadratic residue if and only if $2\sigma_2-P_2$ is.
\end{proof}

Finally we recall the following Lemma on the size of the sphere in $\mathbb{F}_q^2$.   A proof can be found, for example, in \cite{CEHIK}.
\begin{lem}\label{sphere}
Let $S_t = \{x \in \mathbb{F}_q^2 : \| x \|  =t\}$ be the sphere of radius $t \in \mathbb{F}_q^2$.  Let $v$ be the function on $\mathbb{F}_q$ so that $v(0) = q-1$ and $v(t) = -1$ for $t \neq 0$, and let $\left( \frac{\cdot}{q} \right)$ denote the Legendre symbol on $\mathbb{F}_q$ so that
\[
\left( \frac{x}{q}\right) = \left\{\begin{array}{ccc}
1 && x \text{ is a nonzero quadratic residue in } \mathbb{F}_q
\\
0 && x = 0
\\
-1 && \text{ otherwise} 
 \end{array}    \right.
\]
Then $|S_t| = q - \left( \frac{-1}{q} \right) v(t)$.  In particular $|S_t| = q \pm 1$ when $t \neq 0$.
\end{lem}

\section{Sums of unit vectors}
First note that the zero vector can always be written as a sum of two units in a trivial way: 
\[
\left[ \begin{array}{c} 0 \\ \vdots \\ 0 \end{array}\right] = \left[ \begin{array}{c} 1 \\ 0 \\ \vdots \\ 0\end{array}\right] + \left[ \begin{array}{c} -1 \\ 0  \\ \vdots \\ 0 \end{array}\right]
\]
Thus by Witt's Theorem, it is enough to show that for each $L \in \mathbb{F}_q$, some nonzero vector $x \in \mathbb{F}_q^d$ with length $\| x \| = L$ can be written as a sum of two unit vectors.

\subsection{Proof of Proposition \ref{unit-vector,d=2}}
We first show that every vector in $\mathbb{F}_q^2$ is the sum of $2$ unit vectors if and only if $q = 3$.  The result clearly follows from the following propsoition.
\begin{pro} \label{goodcount}
Let $U$ denote the set
\[
U = \left\{L \in \mathbb{F}_q : \text{ every vector in } \mathbb{F}_q^2 \text{ of length } L \text{ is the sum of } 2 \text{ unit vectors}\right\}.
\]
Then, 
\[
|U| = \left\{ \begin{array}{ccc}
\frac{q+3}{2} && q \equiv 3 \pmod{4}
\\
\\
\frac{q-1}{2} && q \equiv 1 \pmod{4}
\end{array}
\right.
\]
\end{pro}
\begin{proof}
First we show the following:
\begin{pro}\label{4L-L^2}
A nonzero vector in $\mathbb{F}_q^2$ of length $L$ is a sum of two unit vectors if and only $4L - L^2$ is a square and $L \neq 0$.
\end{pro}
\begin{proof}
Fix such an $(a,c)$. It is the sum of two unit vectors if and only if there is $(b,d)$ such that the triangle formed by the origin, $(a,c)$ and $(b,d)$ has side lengths $L_1=L, L_2=1, L_3=1$ respectively. As $L_3=L_1+L_2-2\mu$ this corresponds also to $\mu=\frac{L}{2}$.
When $L \neq 0$, Lemma~\ref{lem: grungy1} says such a triangle exists exactly when $L-\frac{L^2}{4}=\frac{4L-L^2}{4}$ is a square (either zero or not) which happens exactly when 
$4L-L^2$ is a square (either zero or not). 
It remains to consider the case $L=0$ which can only occur when $q \equiv 1 \pmod 4$. In this case Lemma~\ref{lem: grungy2} says that such a triangle can exist if and only if 
$\mu=\frac{L}{2}=0$ is nonzero which never happens. Thus a nonzero vector of length zero in the plane cannot be written as the sum of two unit vectors. 
\end{proof}
So we have reduced the problem to determining for what values $q$ the quantity $4L - L^2$ is a nonzero quadratic residue.  By Lemma \ref{sphere} we have
\begin{align*}
|S_2| &= |\{(x,y) \in \mathbb{F}_q^2 : x^2 + y^2 = 4\}|
\\
&= |\{(x,y) \in \mathbb{F}_q^2 : (x-2)^2 + y^2 = 4\}|
\\
&= |\{(x,y) \in \mathbb{F}_q^2 : y^2 = 4x - x^2\}|
\\
&= 2 + \sum_{x \in \mathbb{F}_q \setminus\{0,4\}} \left( \left( \frac{4x - x^2}{q}\right) + 1\right)
\\
&= q + \sum_{x \in \mathbb{F}_q \setminus\{0,4\}} \left( \frac{4x - x^2}{q}\right)
\end{align*}
which shows that
\[
\sum_{x \in \mathbb{F}_q \setminus\{0,4\}} \left( \frac{4x - x^2}{q}\right) = \left\{ \begin{array}{ccc} 1 && q \equiv 3 \pmod{4} \\ -1 && q \equiv 1 \pmod{4} \end{array} \right.
\]
Since $0 \in U$ if and only if $q \equiv 3 \pmod{4}$, this shows that
\[
|U| = \left\{ \begin{array}{ccc} \frac{q+1}{2} + 1 && q \equiv 3 \pmod{4} \\  \\ \frac{q-1}{2} && q \equiv 1 \pmod{4} \end{array} \right.
\]
which completes the proof.

\end{proof}

\subsubsection{Sums of four unit vectors in $\mathbb{F}_q^2$}
Next we show that every vector in $\mathbb{F}_q^2$ is the sum of four unit vectors.  Let $x \in \mathbb{F}_q^2$ have length $\|x\| = L \neq 0$.  Our main tool here is the spectral gap theorem.  Consider the unit distance graph where the vertices are $\mathbb{F}_q^d$, and where two vertices $x, y \in \mathbb{F}_q^d$ are connected if and only if $\| x - y \| = 1$.  Let $E , F \subseteq \mathbb{F}_q^d$.  The spectral gap theorem then asserts that there exists a walk of length $k$ starting at $E$ and ending at $F$ whenever $\sqrt{|E| |F|} \geq \left( \frac{2 \sqrt q} {|S_1| }\right)^k $, where $S_1 = \{x \in \mathbb{F}_q^2 : \|x \| = 1\}$ is the sphere of radius $1$.  By Lemma \ref{sphere} we have
\[
|S_1| = \left\{\begin{array}{ccc}
q -1 && q \equiv 1 \pmod{4}
\\
q + 1 && q \equiv 3 \pmod{4}
 \end{array} \right.
\]
In particular if we take $E = S_1$ and $F = S_L = \{x \in \mathbb{F}_q^2 : \| x \| = L \}$, then $|E| = |F| = |S_1| = q \pm 1$, again by Lemma \ref{sphere}.  When $k = 3$ we can check that
\[
|S_1| \geq \left( \frac{2 \sqrt q} {|S_1| }\right)^3
\]
for all $q \geq 73$.  Thus there exists a walk along the unit-distance graph of length $3$ starting at $S_1$ and ending at $S_L$.  In particular every vector $x \in \mathbb{F}_q^2$ with length $L \neq 0$ is the sum of four unit vectors for $q \geq 73$.  If $L = 0$, then 
\[
|S_0| = \left\{\begin{array}{ccc}
2q -1 && q \equiv 1 \pmod{4}
\\
 1 && q \equiv 3 \pmod{4}
 \end{array} \right.
\]
 Thus if $q \equiv 3 \pmod{4}$, we have
\[
 \left[ \begin{array}{c}  0 \\ 0 \end{array}\right] =  \left[ \begin{array}{c}  1 \\ 0 \end{array}\right] +  \left[ \begin{array}{c}  1 \\ 0 \end{array}\right] +  \left[ \begin{array}{c}  -1 \\ 0 \end{array}\right] +  \left[ \begin{array}{c}  -1 \\ 0 \end{array}\right],
\]
so that every vector is the sum of four unit vectors.  If $q \equiv 1 \pmod{4}$, and since $|S_L| = 2q-1$, then we can apply the Spectral gap theorem again.  There exists a walk of length $3$ from $S_1$ to $S_0$ whenever
\[
\sqrt{|S_1||S_0|} > \left( \frac{2 \sqrt q}{|S_1|} \right)^3
\]
which can be checked to hold when $q \geq 39$.  Finally, it remains to check small values of $q$ by hand, and it turns out that every vector in $\mathbb{F}_q^2$ is the sum of four unit vectors when $q < 73$ as well.  Thus, every vector in $\mathbb{F}_q^2$ is the sum of four unit vectors for all $q$.

Finally it remains to prove that every nonzero vector in $\mathbb{F}_q^2$ is the sum of three unit vectors when $q \equiv 3 \pmod{4}$.  To see this fix $v=\begin{bmatrix} a \\ b \end{bmatrix} \neq \begin{bmatrix} 0 \\ 0 \end{bmatrix}  \in \mathbb{F}_q^2$ of length $\tau$. When $\tau=1$, we may write $v=v+(-v)+v$ to see that $v$ is a sum of three unit vectors. No nonzero vector in the plane has zero length when $q \equiv 3 \pmod 4$ thus we may assume that $\tau \neq 0,1$ for the rest of the proof.  In light of Proposition~\ref{4L-L^2}, $v$ is a sum of three unit vectors if and only if $v=w+u$ where $u$ is a unit vector and $w$ is a vector of length $L$ where $L \neq 0$ and $4L-L^2$ is a square in $\mathbb{F}_q$.  This happens if and only if a triangle of side-lengths $\tau$, $L$ and $1$ exists in the plane $\mathbb{F}_q^2$ for some $L \neq 0$ with $4L-L^2$ a square.  By Corollary~\ref{cor: lengthtype condition for triangles}, this triangle exists if and only if $2\sigma_2(\tau,L,1)-P_2(\tau,L,1)=2(\tau L + L + \tau)-(\tau^2+L^2+1)=-L^2 +(2\tau+2)L-(\tau-1)^2$ is a square in $\mathbb{F}_q$.

Now recalling that $q + 1$ is the size of any circle of nonzero radius in $\mathbb{F}_q^2$, we see that
\[
q + 1 = |\{ (x,y) | y^2+(x-(\tau+1))^2 =4\tau \}| = | \{(x,y) | y^2 = -x^2+(2\tau+2)x-(\tau-1)^2 \}|.
\]
Thus 
\[
q + 1 = \sum_{x \in \mathbb{F}_q} \left( \binom{-x^2+(2\tau+2)x-(\tau-1)^2}{q} + 1\right),
\]
and so
$$
1 = \sum_{x \in \mathbb{F}_q} \binom{-x^2+(2\tau+2)x-(\tau-1)^2}{q}.
$$
We conclude that 
\[
|\{ x \in \mathbb{F}_q | -x^2 + (2\tau+2)x-(\tau-1)^2 \text{ is a nonsquare in } \mathbb{F}_q \}| \leq \frac{q-1}{2}.
\]
On the other hand, by Proposition~\ref{goodcount}, when $q \equiv 3 \pmod{4}$, we have
$$| \{ L \mid L \neq 0, 4L-L^2 \text{ is a square in } \mathbb{F}_q \}| = \frac{q+1}{2}. $$ Thus for every 
$\tau$ value, there must exist at least one $L \neq 0$ such that both $4L-L^2$ and $-L^2+(2\tau +2)L - (\tau-1)^2$ are square in $\mathbb{F}_q$. 
It follows that any nonzero vector $v$ must be the sum of three unit vectors.

We note that for general odd prime power $q$, the zero vector is a sum of three unit vectors if and only if an equilateral triangle of side length $1$ exists in the plane $\mathbb{F}_q^2$.
By Corollary~\ref{cor: lengthtype condition for triangles}, such a triangle exists if and only if $2\sigma_2(1,1,1)-P_2(1,1,1)=3$ is a square in $\mathbb{F}_q$. This in turn happens when $3$ is a square in $\mathbb{F}_p$ or $n$ is even. Finally by quadratic reciprocity, $3$ is a square in $\mathbb{F}_p$ happens if and only if $p \in \{1,3, 11\} \pmod{12}$.

\subsection{Proof of Propositions \ref{unit-vector,d=3} and \ref{a^2+b^2=-1}}
We will show a few items here.  First, we show that every vector of nonzero length is the sum of two unit vectors.  Fix $x \in \mathbb{F}_q^3$ with $\| x \| = L$, say $x = \left[ \begin{array}{c} a \\ b \\ 0 \end{array}\right]$.  Note that if $L = 4$, then $x$ is the sum of two unit vectors as 
\[
 \left[ \begin{array}{c} 2 \\ 0 \\ 0 \end{array}\right] =  \left[ \begin{array}{c} 1 \\ 0 \\ 0 \end{array}\right]+ \left[ \begin{array}{c} 1 \\ 0 \\ 0 \end{array}\right]
\]
So suppose $L \in \mathbb{F}_q \setminus \{0,4\}$.  We need the following result.
\begin{pro}\label{nonzero}
Given $a^2 +b^2 = L \in \mathbb{F}_q^*$, there exists $c, t , u \in \mathbb{F}_q$ such that $c^2 + t^2 +u^2 = 1$ and $(a-c)^2 + (b-t)^2 + u^2= 1$.  That is, for any $L \neq 0$, there exists some $u \in \mathbb{F}_q$ such that there is a triangle of side lengths $L, 1-u^2, 1-u^2$ in the plane $\mathbb{F}_q^2$.

\end{pro}
Accepting this for the moment, we can then write
\[
\left[ \begin{array}{c} a \\ b \\ 0 \end{array}\right] = \left[ \begin{array}{c} c \\ t \\ u \end{array}\right] + \left[ \begin{array}{c} a-c \\ b-t \\ -u \end{array}\right].
\]
So it suffices to prove the Proposition.  

By Corollary~\ref{cor:  lengthtype condition for triangles}, there exists such a triangle if and only if 
$$2\sigma_2-P_2=2(2L(1-u^2)+(1-u^2)^2) - (L^2+2(1-u^2)^2)=4L-L^2-L(2u)^2$$ is a square.  So we must show that for some $u\in \mathbb{F}_q$, the quantity $4L - L^2-4Lu^2$ is a quadratic residue.  First note that we may assume $L \neq 4$ as we have
\[
\left[ \begin{array}{c} 2  \\ 0  \\ 0 \end{array}\right] = \left[ \begin{array}{c} 1  \\ 0  \\ 0 \end{array}\right] + \left[ \begin{array}{c} 1  \\ 0  \\ 0 \end{array}\right].
\]
Then since $L \notin \{0,4\}$, we have $4L -L^2 \neq 0$.  Let $S = \{x^2 : x \in \mathbb{F}_q^*\}$ and $S' = \mathbb{F}_q^* \setminus S$.  Notice that we have either $\{-4L u^2 : u \in \mathbb{F}_q\} =S \cup \{0\}$ or 
$\{-4L u^2 : u \in \mathbb{F}_q\} =S' \cup \{0\}$ since $L \neq 0$.  Now, for a fixed $t \in \mathbb{F}_q$, the translation map $f(x) = x + t$ is injective, and hence we have $f(S \cup \{0\}) \cap S \neq \emptyset$ and $f(S' \cup \{0\}) \cap S \neq \emptyset$ as $|S \cup \{0\}| = |S' \cup \{0\}| > |S'|$.  Taking $t = 4L - L^2$ shows that $4L-L^2 -4Lu^2$ must be a square for some $u$.  This completes the proof.

We next prove Proposition \ref{a^2+b^2=-1}.  Recall we aim to show that if $q \equiv 3 \pmod{4}$, then the nonzero vectors of length zero cannot be written as a sum of two unit vectors, though they can be written as a sum of three unit vectors.  Suppose that $a^2 + b^2 = -1$, so that  $\left[ \begin{array}{c} a  \\ b  \\ 1 \end{array}\right]$ is such a nonzero vector with length zero.  Suppose for a contradiction that there exists some $x, y, z \in \mathbb F_q$ such that 
\[ 
\left[ \begin{array}{c} a  \\ b  \\ 1 \end{array}\right] =  \left[ \begin{array}{c} x  \\ y  \\ z \end{array}\right] +  \left[ \begin{array}{c} a-x  \\ b-y  \\ 1-z \end{array}\right]
\]
and where 
\begin{equation}\label{eq1}
(a-x)^2+(b-y)^2+(1-z)^2=1
\end{equation} and 
\begin{equation}\label{eq2}
x^2+y^2+z^2=1.
\end{equation}
Notice that since $q \equiv 3\pmod{4}$, then $a^2 + b^2 = -1$ implies that $a,b \in \mathbb{F}_q^*$.  Simplifying \eqref{eq1}, and substituting in equation \eqref{eq2}, we see that $x=\frac{z+by}{-a}$.  Plugging this back into \eqref{eq2} we must have $(bz-y)^2=-a^2$, which has no solutions since $-1$ is not a quadratic residue when $q \equiv 3 \pmod{4}$.  To see that $\left[ \begin{array}{c} a  \\ b  \\ 1 \end{array}\right]$ can be written as a sum of three unit vectors simply notice that $\left[ \begin{array}{c} a  \\ b  \\ 0 \end{array}\right]$ has length $-1 \neq 0$, and hence it can be written as a sum of two unit vectors by Proposition \ref{nonzero}.  Thus, 
\[
\left[ \begin{array}{c} a  \\ b  \\ 1 \end{array}\right] = \left[ \begin{array}{c} a  \\ b  \\ 0 \end{array}\right] +  \left[ \begin{array}{c} 0  \\ 0  \\ 1 \end{array}\right] 
\]
is the sum of three unit vectors.

The final piece is to show that when $q \equiv 1\pmod{4}$, then every nonzero vector of length zero can be written as a sum of two unit vectors.  Suppose $a, b \in \mathbb{F}_q^*$ satisfy $a^2 + b^2 = 0$.  Then
\[
\left[ \begin{array}{c} a  \\ b  \\ 0 \end{array}\right] = \left[ \begin{array}{c} a  \\ b  \\ 1 \end{array}\right] +  \left[ \begin{array}{c} 0  \\ 0  \\ -1 \end{array}\right].
\]
This completes the proof of Propositions \ref{unit-vector,d=3} and \ref{a^2+b^2=-1}.

\subsection{Proof of Proposition \ref{unit-vector,d=4}}
Recall that we aim to show that every vector in $\mathbb{F}_q^d$ with $d \geq 4$ can be written as a sum of two unit vectors.  We prove this by cases.

\textbf{Case 1:} Suppose that $L \in \mathbb{F}_q\setminus \{0,4\}$.  Recall that every element in $\mathbb{F}_q$ can be written as a sum of two squares by the pigeonhole principle.  Hence, write $L = a^2 + b^2$.  We will show that $(a,b, 0, \dots , 0)$ can be written as a sum of two unit vectors.  Let $s,t \in \mathbb{F}_q^*$ be such that $s^2 + t^2 = 1 - \frac{L}{4}$ which is always possible as $L \notin \{0,  4\}$.  Hence
\[
\left[ \begin{array}{c} a \\ b \\ 0 \\ 0 \\ 0 \\ \vdots \\ 0 \end{array}\right] = \left[ \begin{array}{c} a/2 \\ b/2 \\ s \\ t \\ 0 \\ \vdots \\ 0 \end{array}\right] + \left[ \begin{array}{c} a/2 \\ b/2 \\ -s \\ -t \\ 0 \\ \vdots \\ 0 \end{array}\right]
\]

\textbf{Case 2:} If $L = 0$, then there exist $a,b , c \in \mathbb{F}_q^*$ such that $a^2 + b^2 + c^2 = 0$.  Hence,
\[
\left[ \begin{array}{c} a \\ b \\ c   \\ 0 \\ \vdots \\ 0 \end{array}\right] = \left[ \begin{array}{c} a/2 \\ b/2 \\ c/2   \\ 1 \\ \vdots \\ 0 \end{array}\right] + \left[ \begin{array}{c} a/2 \\ b/2 \\ c/2   \\ -1 \\ \vdots \\ 0 \end{array}\right] 
\]

\textbf{Case 3:} If $L = 4$, then simply notice that
\[
\left[ \begin{array}{c} 2  \\ 0 \\ \vdots \\ 0 \end{array}\right] = \left[ \begin{array}{c} 1  \\ 0 \\ \vdots \\ 0 \end{array}\right] + \left[ \begin{array}{c} 1  \\ 0 \\ \vdots \\ 0 \end{array}\right].
\]
This completes the high dimensional case $d \geq 4$.

\section{Sums of $2 \times 2$ orthogonal matrices}

We are ready to prove Theorem \ref{O(2)}.  Let $A \in Mat_2(\mathbb{F}_q)$.  Then we can write
\[
A := \left[\begin{array}{cc} a_{11} & a_{12} \\ a_{21} & a_{22} \end{array} \right] = \left[\begin{array}{cc} x & -y \\ y & x \end{array} \right] + \left[\begin{array}{cc} w & z \\ z & -w \end{array} \right]
\]
where $x, y , z, w$ are determined by the system 
\[
\left\{\begin{array}{ccc}
x + w = a_{11}
\\
z-y = a_{12}
\\
y + z = a_{21}
\\
x - w = a_{22}
\end{array}\right.
\]
We first write $\left[\begin{array}{cc} x \\ y \end{array} \right]$ as the sum of four unit vectors:
\[
\left[\begin{array}{cc} x \\ y \end{array} \right] = \sum_{i=1}^4 \left[\begin{array}{cc} u_i \\ v_i \end{array} \right] 
\]
It follows that
\[
\left[\begin{array}{cc} -y \\ x \end{array} \right] = \sum_{i=1}^4 \left[\begin{array}{cc} -v_i \\ u_i \end{array} \right],
\]
and hence
\[
\left[\begin{array}{cc} x & -y \\ y & x \end{array} \right] = \sum_{i=1}^4 \left[\begin{array}{cc} u_i & - v_i \\ v_i & u_i \end{array} \right]
\]
is the sum of four orthogonal matrices.  A similar calculation writes $\left[\begin{array}{cc} w & z \\ z & -w \end{array} \right]$ as the sum of four orthogonal matrices.

Next we show that $6$ orthogonal matrices suffice when $q \equiv 3 \pmod{4}$.  Let $A \in Mat_2(\mathbb{F}_q)$, and write $A = B + C$ where
\[
B = \left[\begin{array}{cc} x & -y \\ y & x \end{array} \right]
\]
and
\[
C = \left[\begin{array}{cc} w & z \\ z & -w \end{array} \right]
\]
If the first column $\left[\begin{array}{cc} x \\ y \end{array} \right] \neq \left[\begin{array}{cc} 0 \\ 0 \end{array} \right] $, then we can write this column as the sum of three unit vectors:
\[
\left[\begin{array}{cc} x \\ y \end{array} \right] = \sum_{i=1}^3 \left[\begin{array}{cc} u_i \\ v_i \end{array} \right],
\]
and hence
\[
B =  \sum_{i=1}^3 \left[\begin{array}{cc} u_i & - v_i\\ v_i  & u_i\end{array} \right]
\]
is the sum of $3$ orthogonal matrices.  The same holds for $C$ if $\left[\begin{array}{cc} w \\ z \end{array} \right] \neq \left[\begin{array}{cc} 0 \\ 0 \end{array} \right] $.
Now if $x = y = 0$, then 
\[
B = \left[\begin{array}{cc} 0 & 0 \\ 0  & 0\end{array} \right]  = I_2 + (-I_2),
\]
where $I_2 \in Mat_2(\mathbb{F}_q)$ is the $2\times 2$ identity matrix.  Since $C$ can always be written as a sum of four orthogonal matrices because of its particular form, then $A = B + C$ is always the sum of six orthogonal matrices in this case.

\section{Sums of $d \times d$ orthogonal matrices}
We now consider sums of $d \times d$ orthogonal matrices for $d \geq 3$.  We will rely on the following observation.
\begin{lem}
Let $d \geq 3$, and suppose that every matrix in $Mat_{d-1}(\mathbb{F}_q)$ can be written as a sum of exactly $r$ orthogonal matrices, where $r$ is even.  Then every matrix in $Mat_n(\mathbb{F}_q)$ whose first column has length $1$ can be written as a sum of exactly $3r$ orthogonal matrices.  When $q \equiv 3 \pmod{4}$, then every matrix in $Mat_3(\mathbb{F}_q)$ can be written as a sum of exactly $9r$ orthogonal matrices.  When $d\geq 4$ or when $d = 3$ and $q \equiv 1 \pmod{4}$, then every matrix in $Mat_n(\mathbb{F}_q)$ can be written as the sum of $6r$ orthogonal matrices.
\end{lem}
This together with Theorem \ref{O(2)} proves Theorem \ref{O(d)}, so it suffices to prove the Lemma.
\begin{proof}
By Witt's theorem, if $v$ is a $d$-dimensional column vector of length $1$, there exists $A \in O(d,q)=\{ B \in Mat_d(\mathbb{F}_q) \mid BB^T=I \}$ such that $Av=(1, 0, 0, \dots, 0)^T$. 

Thus if $C$ is a $d \times d$ matrix whose first column has length one, $AC$ is $O(d,q)$-equivalent to $C$ and has the form 
$$
\begin{bmatrix}
1 & f_{12} & \dots & f_{1d} \\
0 & f_{22} & \dots & f_{2d} \\
\vdots & \vdots & \dots & \vdots \\
0 & f_{d2} & \dots & f_{dd} \\ 
\end{bmatrix}
=
\begin{bmatrix}
0 & 0 & \dots & 0 \\
0 & f_{22} & \dots & f_{2d} \\
\vdots & \vdots & \dots & \vdots \\
0 & f_{d2} & \dots & f_{dd} \\ 
\end{bmatrix}
+
\begin{bmatrix}
1 & f_{12} & \dots & f_{1d} \\
0 & 0 & \dots & 0 \\
\vdots & \vdots & \dots & \vdots \\
0 & 0 & \dots & 0 \\ 
\end{bmatrix} = E+F
$$
By assumption, 
$$
\begin{bmatrix}
f_{22} & \dots & f_{2d} \\
\vdots & \dots & \vdots \\
f_{d2} & \dots & f_{dd} \\ 
\end{bmatrix}
=B_1 + \dots + B_r
$$
for some $B_j \in O(d-1,q)$, $1 \leq j \leq r$, and so,
$$
E=\begin{bmatrix}
0 & 0 & \dots & 0 \\
0 & f_{22} & \dots & f_{2d} \\
\vdots & \vdots & \dots & \vdots \\
0 & f_{d2} & \dots & f_{dd} \\ 
\end{bmatrix}
=\begin{bmatrix}
1 & \hat{0}^T \\
\hat{0} & B_1 \\
\end{bmatrix}
+
\begin{bmatrix}
-1 & \hat{0}^T \\
\hat{0} & B_2 \\
\end{bmatrix}
+ \dots 
+
\begin{bmatrix}
1 & \hat{0}^T \\
\hat{0} & B_{r-1} \\
\end{bmatrix}+
\begin{bmatrix}
-1 & \hat{0}^T \\
\hat{0} & B_r \\
\end{bmatrix}$$
where $\hat{0}$ is the $(d-1)$ dimensional zero vector. Note the matrices on the right hand side of this equation are all in $O(d,q)$, and so $E$ is the sum of exactly $r$ orthogonal matrices.

The matrix $F$ can be written as $G+H$ where $G=\begin{bmatrix}
0 & f_{12} & \dots & f_{1d} \\
0 & 0 & \dots & 0 \\
\vdots & \vdots & \dots & \vdots \\
0 & 0 & \dots & 0 \\ 
\end{bmatrix}$
and $H=\begin{bmatrix}
1 & 0 & \dots & 0 \\
0 & 0 & \dots & 0 \\
\vdots & \vdots & \dots & \vdots \\
0 & 0 & \dots & 0 \\ 
\end{bmatrix}$.
As permutation matrices are in $O(d,q)$, $G$ is $O(d,q)$-equivalent to the matrix $\begin{bmatrix}
0 & 0 & \dots & 0 \\
0 & f_{12} & \dots & f_{1d} \\
0 & 0 & \dots & 0 \\
\vdots & \vdots & \dots & \vdots \\
0 & 0 & \dots & 0 \\ 
\end{bmatrix}$
which can be written as the sum of $r$ orthogonal matrices by exactly the same reasoning used for $E$ earlier. Thus $G$ itself is the sum of $r$ orthogonal matrices.

Finally $H$ is $O(d,q)$-equivalent (permute row 1 and 2 and then column 1 and column 2) to 
$\begin{bmatrix}
0 & 0 & \dots & 0 \\
0 & 1 & \dots & 0 \\
0 & 0 & \dots & 0 \\
\vdots & \vdots & \dots & \vdots \\
0 & 0 & \dots & 0 \\ 
\end{bmatrix}$
and so is also the sum of $r$ orthogonal matrices.

Putting this all together, we see that the general $d \times d$ matrix with first column of length one that we started with, is a sum of exactly $3r$ orthogonal matrices.

When $d \geq 4$ or ($d=3$ and $q=1 \pmod 4$), every vector can be written as a sum of two unit vectors. 
Thus every $d \times d$ matrix can be written as the sum of two matrices whose first column has length one, and hence by the previous result, as a sum of exactly $6r$ orthogonal matrices.

When $d=3$ and $q=3 \pmod 4$, every vector can either be written as the sum of two or three unit vectors. 
Thus every $d \times d$ matrix can be written as either the sum of two or three matrices whose first column has length one, and thus as a sum of exactly $6r$ or $9r$ orthogonal matrices.
In the case that the sum uses $6r$ orthogonal matrices, it can be extended to $9r$ orthogonal matrices by just adding an additional $\frac{3r}{2}$ copies of the identity matrix $I$ and 
$\frac{3r}{2}$ copies of $-I$, as $r$ is even.

\end{proof}

\section{Proof of Proposition \ref{triangles}}
Recall that
\[
SO(2;q) = \{A \in O(2;q) : \det(A) = 1\} = \left\{ \begin{bmatrix} a & -b \\ b & a  \end{bmatrix}: a^2 + b^2 = 1\right\},
\]
and thus $|SO(2,q)|=|S_1| = q \pm 1$ by Lemma \ref{sphere}.  Thus 
\[
|O(2;q)|=2|SO(2;q)| = 2(q \pm 1)
\]
depending on whether $q \equiv 1 \pmod{4}$ or $q\equiv 3\pmod{4}$.  By Proposition \ref{pro:trianglefacts} the congruence classes of nondegenerate triangles in the plane $\mathbb{F}_q^2$ are in one-to-one correspondence with the right cosets of $O(2;q)$ in $GL_2(q)$.  Using that $|GL_2(q)|=(q^2-1)(q^2-q)=q(q+1)(q-1)^2$, the number of distinct congruence classes is then
\[
\frac{|GL_2(\mathbb{F}_q)|}{|O(2;q)|} = \frac{q(q+1)(q-1)^2}{2(q \pm 1)},
\]
and the Theorem follows.

\section{Eigenvalues of the $O(2)$-graph}

Recall that our proof on the diameter of the $O(2)$ graph did not directly appeal to the spectral gap theorem applied to the $O(2)$ graph.  Rather, we used the spectral gap theorem to determine the minimum number $k$ such that we can write every vector in $\mathbb{F}_q^d$ as a sum of $k$ unit vectors.  Nonetheless, it is still be useful to categorize the spectrum of the $O(2)$ graph, so we do so below.  

First consider an eigenvalue $\lambda_A =  \sum_{g \in O(2)} \chi(Tr(Ag))$ when $A \in GL_2$.  Note $\lambda_A = \sum_{g \in O(2)} \chi(Tr(gA)) = \sum_{t \in O(2)A} \chi(Tr(t))$ where the last sum is a sum over the right coset $O(2)A$ or in other words, over the set of all origin pinned triangles, congruent to the $A$-triangle.

By Theorem~\ref{thm: nondegenerate triangle congruence}, we have that if the columns of $\mathbb{A}$ have lengths $L_1, L_2$ and dot product $\mu$, then this can be written:

\[
\lambda_A=\sum_{\begin{bmatrix} a \\ c \end{bmatrix} \in S_1(L_1)} \sum_{\begin{bmatrix} b \\ d \end{bmatrix} \in S_1(L_2), ab+cd=\mu} \chi(a+d),
\]

and furthermore $L_1L_2-\mu^2$ is a nonzero square.  As this eigenvalue only depends on $(L_1, L_2, \mu)$ we will also denote it by $\lambda_{L_1,L_2,\mu}$. Note $\lambda_{L_1,L_2,\mu} = \lambda_{L_2,L_1,\mu}$ due to the fact 
$A$ and $A\begin{bmatrix} 0 & 1 \\ 1 & 0 \end{bmatrix}$ determine the same $O(2)$ double coset, and hence $\lambda_A=\lambda_{A\begin{bmatrix} 0 & 1 \\ 1 & 0 \end{bmatrix}}$.

When $L_1 \neq 0$, using Lemma~\ref{lem: grungy1}, this becomes:

\[
\lambda_{L_1,L_2,\mu} = \sum_{\begin{bmatrix} a \\ c \end{bmatrix} \in S_1(L_1)} \left(\chi\left(a + \frac{\mu}{L_1}c + \frac{\sqrt{L_1L_2-\mu^2}}{L_1}a\right) + \chi\left(a + \frac{\mu}{L_1}c - \frac{\sqrt{L_1L_2-\mu^2}}{L_1}a\right)\right).
\]

If we let $\widehat{F}$ denote the unnormalized Fourier transform of the indicator function of $S_1(L_1)$ (see appendix), this translates to:

\[
\lambda_{L_1,L_2,\mu} = \widehat{F}\left(1+\frac{\sqrt{L_1L_2-\mu^2}}{L_1}, \frac{\mu}{L_1}\right) + \widehat{F}\left(1-\frac{\sqrt{L_1L_2-\mu^2}}{L_1}, \frac{\mu}{L_1}\right).
\]

It is well known (\cite{IR}) that $|\widehat{F}(\alpha,\beta)| \leq 2 \sqrt{q}$ as long as $(\alpha, \beta) \neq (0,0)$. It then follows that as long as 
either $\mu \neq 0$ or $L_1 \neq L_2$, that 
$|\lambda_{L_1,L_2,\mu}| \leq 4\sqrt{q}$.

Note also that unfortunately

\[
\lambda_{L,L,0} = \widehat{F}(0,0) + \widehat{F}(2,0)=q \pm 1 +\widehat{F}(2,0), \text{ when } L \neq 0.
\]

When $L_1 =0$ using Lemma~\ref{lem: grungy2} we get that $\mu \neq 0$ and 

\[
|\lambda_{0,L_2,\mu}|= |\lambda_{L_2,0,\mu}| \leq 4\sqrt{q}, \text{ when } L_2 \neq 0
\]
Finally
\[
\lambda_{0,0,\mu}=\sum_{\begin{bmatrix} a \\ ia \end{bmatrix} \in S_1(0)} \sum_{\begin{bmatrix} b \\ -ib \end{bmatrix} \in S_1(0), 2ab=\mu} (\chi(a-ib) + \chi(a+b))
\]
which becomes
\[
\lambda_{0,0,\mu}=\sum_{a \neq 0} \chi(a - \frac{i\mu}{2a}) + \sum_{a \neq 0} \chi(a+\frac{\mu}{2a})
\]
 which is a sum of two nontrivial Kloosterman sums and hence $|\lambda_{0,0,\mu}| \leq 4 \sqrt{q}$ also (for all nonzero $\mu$).

We summarize these results as follows:

\begin{pro}
\label{pro: nondegenerate eigenvalues}
Fix $q$ an odd prime power and let $A \in GL_2(q)$ be an invertible matrix with first column length $L_1$, second column length $L_2$ and dot product between the two columns $\mu$.
Then $L_1L_2 - \mu^2$ is a nonzero square in $\mathbb{F}_q$. The corresponding eigenvalue $\lambda_A$ of the $O(2)$-graph only depends on $L_1, L_2, \mu$ and will be denoted 
$\lambda_{L_1,L_2,\mu}$. It satisfies $\lambda_{L_1,L_2,\mu}=\lambda_{L_2,L_1,\mu}$ and the following: \\
$|\lambda_{L_1,L_2,\mu}| \leq 4\sqrt{q}$ in all cases except when $L_1=L_2 \neq 0, \mu = 0$. In the last case we have 
$\lambda_{L,L,0} = |SO(2)| + E$ where $|E| \leq 2 \sqrt{q}$ and $|SO(2)| = q \pm 1$. 
\end{pro}

Given Proposition~\ref{pro: nondegenerate eigenvalues} and the fact that $\lambda_{0} = |O(2;q)|=2(q \pm 1)$, it remains only to study $\lambda_A$ when $A$ is a $2 \times 2$ matrix of rank one. We need to recall the following fact (Theorem 2.2 in \cite{IR}) as we need it for eigenvalue computation:

\begin{lem} \label{2timesSquareRootOfq bound}
Let $(a,b) \neq (0,0)$. Then, $\left| \sum_{(x,y) \in S_1} \chi(ax+by) \right | < 2 \sqrt{q}$. 
\end{lem}

As we stated earlier, we want to study $\lambda_A$ when $A$ is a $2 \times 2$ matrix of rank one. By direct computation, if $A$ and $B$ belong to same double coset of $O(2)$ in $Mat_2(\mathbb F_q)$ meaning if $B \in O(2)AO(2)$, then $\lambda_{A}=\lambda_{B}$. Hence, we need to determine double cosets of $O(2)$. Notice that any $2 \times 2$ matrix of rank one can easily be written as the outer product of two vectors, i.e. $A=w v^{T}$ for some column $2$-vectors $w,v$. If $A=\begin{bmatrix}
      	a  & b\\
      	sa & sb \\
      	\end{bmatrix}$ for some $(a,b) \neq (0,0)$ and $s \in \mathbb F_q$, we will use $w=\begin{bmatrix}
      	1  \\
      	s \\
      	\end{bmatrix}$ and $v=\begin{bmatrix}
      	a  \\
      	b \\
      	\end{bmatrix}$ for our computations. Otherwise, $A$ should be in the form of $\begin{bmatrix}
      	0  & 0\\
      	a &  b \\
      	\end{bmatrix}$ for some $(a,b) \neq (0,0)$, and we will use $w=\begin{bmatrix}
      	0  \\
      	1 \\
      	\end{bmatrix}$, $v=\begin{bmatrix}
      	a  \\
      	b \\
      	\end{bmatrix}$. This helps us a lot since $O(2)AO(2)= O(2) (w v^{T}) O(2)=(O(2) w) (O(2) v)^{T}$, and since each orbit of $O(2)$ action on $\mathbb F_q^2$ consists of all vectors with the same length. In short, if $A=w_1 v_1^{T}$ and $B=w_2 v_2^{T}$ and if $\|w_1\|= \|w_2 \|$ and $\|v_1\|= \|v_2\|$, then $\lambda_{A}=\lambda_{B}$. Now, we are ready to prove the following:

\begin{pro}
Let $A$ be a $2 \times 2$ matrix of rank one. If $a \neq \pm bi$, then $\left| \lambda_{A} \right | \leqslant 4 \sqrt{q}$. If $a=bi$ or $a=-bi$, then \[ \left| \lambda_{A} \right | \leqslant (q \pm 1) + 2 \sqrt{q} \] $+$ where $q =3 \text{ mod } 4$, and $-$ otherwise. \end{pro}

\begin{proof}
First, notice that $A=\begin{bmatrix}
      	a  & b\\
      	0 &  0 \\
      	\end{bmatrix}$ and $B=\begin{bmatrix}
      	0  & 0 \\
      	a  & b \\
      	\end{bmatrix}$ are corresponding to the same eigenvalue, since $\|w_1\|= \|w_2 \|$ and $\|v_1\|= \|v_2\|$ (where $w_i, v_i$'s are defined as above). Hence, we need to calculate $\lambda_{A}$ only for $A=\begin{bmatrix}
      	a   & b\\
      	sa & sb \\
      	\end{bmatrix}$ where $(a,b) \neq (0,0)$ and $s \in \mathbb F_q$. 
	
	Second, let's calculate the bound for $\lambda_{A}$ for a special case, $s=0$ case, since it is easier. Hence, we let $A=\begin{bmatrix}
      	a  & b\\
      	0 &  0 \\
      	\end{bmatrix}$ where $(a,b) \neq (0,0)$ and $a^2+b^2=L$ for some $L \in \mathbb F_q$. Then, $\lambda_{A}= \sum_{g \in O(2)} \chi(Tr(gA))$. Recall $g$ should be $\begin{bmatrix}
      	x  & y\\
      	y & -x \\
      	\end{bmatrix}$ or $\begin{bmatrix}
      	x & -y\\
      	y & x \\
      	\end{bmatrix}$ for some $(x,y) \in S_1$, hence $\lambda_{A}=2 \sum_{(x,y) \in S_1} \chi(ax+by)$ where $(a,b) \neq (0,0)$ and $a^2+b^2=L$. By Lemma~\ref{2timesSquareRootOfq bound}, we have $|\lambda_{A}|< 4 \sqrt{q}$.

Now, let's calculate the bound for $\lambda_{A}$ when $s \in \mathbb F_q^{\ast}$. If $A=\begin{bmatrix}
      	a   &  b\\
      	sa & sb \\
      	\end{bmatrix}$ for some $(a,b) \neq (0,0)$ and $s \in \mathbb F_q^{\ast}$, we have \[ \lambda_{A}= \sum_{g \in O(2)} \chi(Tr(gA)) = \sum_{(x,y) \in S_1} \chi(ax+ays+by-bxs) \ + \sum_{(x,y) \in S_1} \chi(ax-ays+by+ bxs).\] Hence,	
\[ \left| \lambda_{A}  \right | \leqslant \left| \sum_{(x,y) \in S_1} \chi(ax+ays+by-bxs) \right |+ \left| \sum_{(x,y) \in S_1} \chi(ax-ays+by+ bxs) \right |.\]
If $a \neq \pm bi$, then $\left| \lambda_{A} \right |$ is bounded by $4 \sqrt{q}$ by Lemma~\ref{2timesSquareRootOfq bound}. If $a=bi$ or $a=-bi$, then either $ax+ays+by-bxs$ or $ax-ays+by+ bxs$ vanishes and we have $\left| \lambda_{A}  \right |$ is bounded by $(q \pm 1)+ 2 \sqrt{q}$ where the sign of $\pm$ depends on whether $q = 3 \text{ mod } 4$ or $q=1 \text{ mod } 4$. \end{proof}

\end{document}